\documentclass[12pt]{article}

\usepackage{amsmath,amssymb,amsthm,amsfonts,mathtools,dsfont,hyperref,fullpage,
	xcolor,indentfirst}
\usepackage[english]{babel}

\hypersetup{colorlinks,linkcolor={red!50!black},citecolor={blue!50!black}}

\numberwithin{equation}{section}

\newtheorem{theorem}{Theorem}[section]
\newtheorem{lemma}{Lemma}[section]

\theoremstyle{remark}

\newtheorem{remark}{Remark}[section]

\DeclareMathOperator{\supp}{supp}
\DeclareMathOperator{\Leb}{Leb}

\begin{document}
	
\title{Convergence of point processes associated with coupon collector's and Dixie cup problems}
	
\author{Andrii Ilienko\\
	{\small\it Igor Sikorsky Kyiv Polytechnic Institute}\\
	{\small\it ilienko@matan.kpi.ua}}
	
\date{\today}
	
\maketitle
	
\begin{abstract}
 We prove that, in the coupon collector's problem, the point processes given by the times of $r^{th}$ arrivals for coupons of each type, centered and normalized in a proper way, converge toward a non-homogeneous Poisson point process. This result is then used to derive some generalizations and infinite-dimensional extensions of classical limit theorems on the topic.
\end{abstract}
	
\section{Introduction}

The coupon collector's problem (CCP), as well as its generalization known as the Dixie cup problem (DCP), belong to the classics of combinatorial probability. Their statements are as follows: a person collects coupons, each of which belongs to one of $n$ different types. The coupons arrive one by one at discrete times, the type of each coupon being equiprobable and independent of types of preceding ones. Let $T_c^{(n)}$ stand for the (random) number of coupons a person needs to collect in order to assemble $c\in\mathbb N$ complete collections. The most typical questions concern asymptotics of $\mathbb ET_c^{(n)}$ and distributional limit theorems for $T_c^{(n)}$ themselves as $n\to\infty$. Sometimes, the case $c=1$ refers to CCP while $c\ge2$ to DCP. It should be noted that the terminology is not well established in the literature: sometimes both problems are attributed as CCP or, on the contrary, as DCP. The above terminology follows \cite{Hol86} and \cite{FL}.

CCP, DCP and their further generalizations have a long history, going back to de Moivre, Euler and Laplace. Since the 60s of the past century, there has appeared an extensive literature on the topic. In particular, we recall here a classical result by Erd\H{o}s and R\'enyi \cite{ER}:
\begin{gather}
\label{ER_E}
\mathbb ET_c^{(n)}=n\ln n+(c-1)n\ln\ln n+
(\gamma-\ln (c-1)!)n+{\scriptstyle\mathcal{O}}(n),\\
\lim_{n\to\infty}\mathbb P\Bigl\{\frac{T_c^{(n)}}n-\ln n-(c-1)\ln\ln n<x\Bigr\}=\exp\Bigl\{-\frac{\mathrm e^{-x}}{(c-1)!}\Bigr\},
\label{ER}
\end{gather}
with $\gamma=-\Gamma'(1)$ standing for the Euler–Mascheroni constant.

Subsequently, the theory was developed and generalized in different directions: non-equal probabilities of coupon types (plenty of literature --- \cite{Neal}, \cite{DP13}, \cite{ABS}, to cite just a few), various random sceneries (\cite{Hol01}, \cite{DP14}, \cite{FL}), collecting pairs (\cite{Mlad}, \cite{GM}), and so on. A nice and quite elementary introduction to the topic is given in \cite{FS}.

In his seminal paper, Holst \cite{Hol86} proposed a fruitful poissonization idea which allowed to prove limit results like \eqref{ER_E}, \eqref{ER} avoiding intricate combinatorial calculations. In a very recent paper by Glava\v{s} and Mladenovi\'c \cite{GM}, the connections between CCP and Poisson processes were shown to be even more tight. As a matter of fact, it was proved that the point processes given by the times of first arrivals for coupons of each type, centered and normalized in a proper way, converge toward a non-homogeneous Poisson point process as $n\to\infty$. The above convergence is, as usual, understood as the distributional one in the space of all locally finite point measures, endowed with the vague topology. The proof is based on rather delicate combinatorial arguments. As for the DCP, the authors do not consider the corresponding results, confining themselves to just pointing out that, within the framework of their methods, the relevant formulations and proofs would require much more technical details.

Inspired by this paper, the present note pursues a threefold objective. Firstly, we generalize the above result to the case of DCP. Secondly, to this end, we develop a specific approach involving a poissonization technique in the spirit of \cite{Hol86} and some coupling-based depoissonization procedure. This allows for avoiding sophisticated combinatorial machinery used in \cite{GM}. Thirdly, we demonstrate the power of this result from the applications point of view. It can be used to easily derive some generalizations and infinite-dimensional extensions of classical limit theorems on the topic.

\section{Preliminaries and notation}

Let $Y_{i,r}^{(n)}$, $i\in\mathbb N_n=\{1,\ldots,n\}$, $r,n\in\mathbb N$, stand for the time the $r^{\rm th}$ coupon of type $i$ arrives. So, $Y_{i,r}^{(n)}\sim\mathtt{NegBin}\bigl(r,\frac 1n\bigr)$, where by $\mathtt{NegBin}$ we mean that version of the negative binomial distribution which counts trials up to (and including) the $r^{th}$ success:
\[\mathbb P\bigl\{Y_{i,r}^{(n)}=k\bigr\}=
\binom{k-1}{r-1}\Bigl(\frac 1n\Bigr)^r\Bigl(1-\frac 1n\Bigr)^{k-r},\quad k\ge r.\]
For fixed $n$ and $r$, the random variables $Y_{i,r}^{(n)}$, $i\in\mathbb N_n$, are identically distributed but not independent, since $Y_{i,r}^{(n)}\ne Y_{i',r}^{(n)}$ for $i\ne i'$.

In order to cope with this dependency, following \cite{Hol86}, we consider a poissonized scheme. That is, we assume that coupons arrive at random times with independent $\mathtt{Exp}(1)$-distributed intervals $E_j$, $j\in\mathbb N$. More formally, we
introduce the unit-rate independently marked Poisson point process
$\Xi^{(n)}=\sum_{k=1}^\infty\delta_{(X_k,M_k)}$ with the uniform on $\mathbb N_n$ mark distribution: $\mathbb P\{M_k=i\}=\frac 1n$, $i\in\mathbb N_n$. Here, $X_k=\sum_{j=1}^kE_j$ stand for the arrival times, $M_k$ for the types of arriving coupons, and $\delta_u$ for the Dirac measure $\mathds1\{u\in\cdot\}$. Hence,
by Theorem 5.8 in \cite{LP}, \[\Xi_i^{(n)}=\sum_{\substack{(X_k,M_k)\in\\\supp\Xi^{(n)}}}\delta_{X_k}\mathds 1\{M_k=i\}=\Xi^{(n)}(\cdot\times\{i\}),\quad i\in\mathbb N_n,\]
are independent $\frac 1n$-rate Poisson point processes. The process $\Xi_i^{(n)}$ describes arrivals of coupons of $i^{th}$ type. In this setting, the random variables $Y_{i,r}^{(n)}$ introduced at the beginning of the section admit the following representation:
\begin{equation}
\label{Y}
Y_{i,r}^{(n)}=\min\Bigl\{m\ge r\colon\sum_{k=1}^m\mathds 1\{M_k=i\}=r\Bigr\}.
\end{equation}

Let $Z_{i,r}^{(n)}$, $i\in\mathbb N_n$, $r,n\in\mathbb N$, stand for the time the $r^{\rm th}$ coupon of type $i$ arrives in the above poissonized scheme. So, $Z_{i,r}^{(n)}$ are independent gamma-distributed random variables, $Z_{i,r}^{(n)}\sim\mathtt{\Gamma}\bigl(r,\frac 1n\bigr)$. For any fixed $n$, the sequences $\bigl(Y_{i,r}^{(n)}\bigr)$ and $\bigl(Z_{i,r}^{(n)}\bigr)$ are now given on a common probability space and coupled by
\begin{equation}
\label{ZY}
Z_{i,r}^{(n)}=\sum_{j=1}^{Y_{i,r}^{(n)}}E_j,\quad i\in\mathbb N_n,\,r\in\mathbb N.
\end{equation}
Moreover, $\bigl(Y_{i,r}^{(n)},i\in\mathbb N_n, r\in\mathbb N\bigr)$ is independent of
$(E_j, j\in\mathbb N)$, since, by \eqref{Y}, the former sequence is determined solely by marks $M_k$.

Let us denote
\begin{equation}
\label{psi}
\psi_r^{(n)}(x)=\frac xn-\ln n-(r-1)\ln\ln n,\quad x\in\mathbb R.
\end{equation}
The main object of our study is the (centered and normalized by means of $\psi_r^{(n)}$) point process of $r^{\rm th}$ arrivals of different types:
\begin{equation}
\label{xin_def}
\xi_r^{(n)}=\sum_{i=1}^n\delta_{\psi_r^{(n)}\bigl(Y_{i,r}^{(n)}\bigr)}.
\end{equation}
In what follows, we will also need the counterpart of $\xi_r^{(n)}$ in the poissonized setting:
\[\eta_r^{(n)}=\sum_{i=1}^n\delta_{\psi_r^{(n)}\bigl(Z_{i,r}^{(n)}\bigr)}.\]

\section{Main result}

Before proceeding to the main result, we recall some basic definitions related to convergence of point processes (see \cite{Res87}, \cite{Res07}, or \cite{Kal} for details). Let $M_p(\mathbb R)$ denote the space of all locally finite point measures on $\mathbb R$. For $\mu,\mu_1,\mu_2,\ldots\in M_p(\mathbb R)$, $\mu_n$ are said to converge vaguely to $\mu$ (denoted as $\mu_n\xrightarrow{v}\mu$) if $\int_{\mathbb R}f\,\mathrm d\mu_n\to\int_{\mathbb R}f\,\mathrm d\mu$ for each continuous compactly supported test function $f\colon\mathbb R\to[0,+\infty)$.
The set $M_p(\mathbb R)$, endowed with the corresponding topology, can be metrized as a complete separable metric space. This setting allows to consider the distributional convergence of point processes $\xi,\xi_1,\xi_2\ldots$, denoted as $\xi_n\xrightarrow{vd}\xi$. The main result of this note, Theorem \ref{main_th} below, asserts that the point processes $\xi_r^{(n)}$ converge in this sense toward a non-homogeneous Poisson process.

\begin{theorem}
	\label{main_th}
	Let $\xi_r$	be the Poisson point process on $\mathbb R$ with intensity measure $\lambda_r({\rm d}x)=\frac 1{(r-1)!}\mathrm{e}^{-x}\,{\rm d}x$. Then $\xi_r^{(n)}\xrightarrow{vd}\xi_r$ as $n\to\infty$.
\end{theorem}

\begin{remark}
	\label{nhg_PPP}
	The limiting process $\xi_r$ allows for a simple interpretation. Let $\xi$ be a stationary unit-rate Poisson point process restricted to $(0,+\infty)$, and put
	\begin{equation}
	\label{h}
	h(x)=-\ln(r-1)!-\ln x,\quad x>0.
	\end{equation}
	Then, $\xi_r\overset d=\xi\circ h^{-1}$. In other words, $\xi_r\overset d=\sum_{x\in\supp\xi}\delta_{h(x)}$.
	
	Indeed, by the mapping theorem (see, e.g., Theorem 5.1 in \cite{LP}), $\xi\circ h^{-1}$ is a Poisson process with intensity measure of the form $\Leb\circ h^{-1}$, where $\Leb$ stands for the Lebesgue measure. Since, for any $[a,b]\subset\mathbb R$,
	\[\bigl(\Leb\circ h^{-1}\bigr)[a,b]=\Leb\biggl[\frac{\mathrm e^{-b}}{(r-1)!},\frac{\mathrm e^{-a}}{(r-1)!}\biggr]=\int_a^b\frac 1{(r-1)!}\mathrm{e}^{-x}\,{\rm d}x=\lambda_r[a,b],\]
	the result follows.
\end{remark}

We divide the proof of Theorem \ref{main_th} into several steps. First, we prove a similar result in the poissonized setting.

\begin{lemma}
	\label{lem_pois}
	We have $\eta_r^{(n)}\xrightarrow{vd}\xi_r$ as $n\to\infty$.
\end{lemma}

\begin{proof}
	Since $Z_{i,r}^{(n)}\sim\mathtt{\Gamma}\bigl(r,\frac 1n\bigr)$, the density $f_{i,r}^{(n)}$ of $Z_{i,r}^{(n)}$ is
	\[f_{i,r}^{(n)}(x)=\frac 1{n^r(r-1)!}x^{r-1}\mathrm e^{-\frac xn}\mathds1\{x\ge0\}.\]
	So, the density $\tilde f_{i,r}^{(n)}$ of $\psi_r^{(n)}\bigl(Z_{i,r}^{(n)}\bigr)$ takes the form
	\begin{multline*}
	\tilde f_{i,r}^{(n)}(x)=nf_{i,r}^{(n)}(nx+n\ln n+(r-1)n\ln\ln n)\\=
	\frac 1{n(r-1)!}\mathrm e^{-x}\Bigl(1+(r-1)\frac{\ln\ln n}{\ln n}+\frac x{\ln n}\Bigr)^{r-1}\mathds1\{x\ge-\ln n-(r-1)\ln\ln n\}.
	\end{multline*}
	Hence, for any $x\in\mathbb R$,
	\[\lim_{n\to\infty}n\tilde f_{i,r}^{(n)}(x)=\frac 1{(r-1)!}\mathrm e^{-x},\]
	and, moreover, this convergence is uniform over bounded sets. Then, for each such Borel set $B$, as $n\to\infty$,
	\[n\int_B\tilde f_{i,r}^{(n)}(x)\,\mathrm dx\to\int_B\frac 1{(r-1)!}\mathrm e^{-x}\,\mathrm dx,\]
	and so, by Proposition 3.12 in \cite{Res87},
	\[n\mathbb P\{\psi_r^{(n)}\bigl(Z_{i,r}^{(n)}\bigr)\in\cdot\}\xrightarrow{v}
	\lambda_r(\cdot)\quad\text{on $\mathbb R$}.\]
	Taking into account the independence of $Z_{i,r}^{(n)}$ for different $i$, the
	well-known fact on the convergence of binomial point processes toward a Poisson one (see, e.g., the warm-up in the proof of Proposition 3.21 in \cite{Res87}) delivers the claim.
\end{proof}

In the next stage, we will need some depoissonization procedure in order to turn $\eta_r^{(n)}\xrightarrow{vd}\xi_r$ into $\xi_r^{(n)}\xrightarrow{vd}\xi_r$. Such depoissonization techniques usually involve bounds on distances between random elements in poissonized and depoissonized settings (see, e.g., Lemma 1.4 in \cite{Pen} for random variables and, as its application, Theorem 3.2 in \cite{MM} for point processes). Since $\xi_r^{(n)}$ and $\eta_r^{(n)}$ are related by \eqref{ZY}, we may use this idea and rate how close both processes are.

\begin{lemma}
	For any segment $[a,b]\subset\mathbb R$ and any $n\in\mathbb N$, $\varepsilon>0$,
	\begin{multline}
	\label{nebound}
	\mathbb P\bigl\{\xi_r^{(n)}[a,b]\ne\eta_r^{(n)}[a,b]\bigr\}\\\le
	c_r\varepsilon^{-4}n^{-1}+
	\mathbb P\bigl\{\eta_r^{(n)}[a-\varepsilon,a+\varepsilon]\ge1\bigr\}+
	\mathbb P\bigl\{\eta_r^{(n)}[b-\varepsilon,b+\varepsilon]\ge1\bigr\}
	\end{multline}
	with some $c_r>0$.
\end{lemma}

\begin{proof} The idea of \eqref{nebound} is pretty simple. Roughly speaking, there may be two reasons for $\xi_r^{(n)}[a,b]\ne\eta_r^{(n)}[a,b]$: either some point of $\xi_r^{(n)}$ deviated far away from the corresponding point of $\eta_r^{(n)}$, or there are points of $\eta_r^{(n)}$ close enough to the boundary of $[a,b]$. The first term on the right-hand side of \eqref{nebound} is responsible for the first reason while the rest for the second one. We proceed to the implementation.
	
	Fix an $\varepsilon>0$. Then, by \eqref{psi} and \eqref{ZY},
	\begin{equation*}
	\mathbb P\bigl\{\bigl|\psi_r^{(n)}\bigl(Z_{i,r}^{(n)}\bigr)-
	\psi_r^{(n)}\bigl(Y_{i,r}^{(n)}\bigr)\bigr|>\varepsilon\bigr\}=
	\mathbb P\bigl\{\bigl|Z_{i,r}^{(n)}-
	Y_{i,r}^{(n)}\bigr|>\varepsilon n\bigr\}=
	\mathbb P\Biggl\{\Biggl|\sum_{j=1}^{Y_{i,r}^{(n)}}(E_j-1)\Biggr|>\varepsilon n\Biggr\}
	\end{equation*}
	for each $i\in\mathbb N_n$. So, by Markov inequality, the latter does not exceed 
	\begin{equation*}
	\Delta_{\varepsilon,n}=(\varepsilon n)^{-4}\,\mathbb E\Biggl(\sum_{j=1}^{Y_{i,r}^{(n)}}(E_j-1)\Biggr)^4.
	\end{equation*}
	Since $E_j-1$ are centered i.i.d., $\Delta_{\varepsilon,n}$ can be easily calculated by a standard conditioning argument: we have
	\begin{equation*}
	\Delta_{\varepsilon,n}=(\varepsilon n)^{-4}\bigl(m_4\,\mathbb EY_{1,r}^{(n)}+
	m_2^2\,\mathbb EY_{1,r}^{(n)}(Y_{1,r}^{(n)}-1)\bigr),
	\end{equation*}
	with $m_k$ standing for $\mathbb E(E_1-1)^k$. As $Y_{1,r}^{(n)}\sim\mathtt{NegBin}\bigl(r,\frac 1n\bigr)$, straightforward calculations lead to the bound
	\begin{equation*}
	\Delta_{\varepsilon,n}\le c_r\varepsilon^{-4}n^{-2},\quad n\in\mathbb N,
	\end{equation*}
	with some $c_r>0$. Now we can finally bound the probability that some point of $\xi_r^{(n)}$ is far away from the corresponding point of $\eta_r^{(n)}$: by subadditivity,
	\begin{equation*}
	\mathbb P\bigl\{\bigl|\psi_r^{(n)}\bigl(Z_{i,r}^{(n)}\bigr)-
	\psi_r^{(n)}\bigl(Y_{i,r}^{(n)}\bigr)\bigr|>\varepsilon\text{ for some $i\in\mathbb N_n$}\bigr\}\le n\Delta_{\varepsilon,n}\le c_r\varepsilon^{-4}n^{-1}.
	\end{equation*}
	So, \eqref{nebound} follows from the reasoning at the beginning of the proof.	
\end{proof}

In order to deal with the last two terms on the right-hand side of \eqref{nebound}, we will need an easy technical lemma on the distributional convergence of point processes.

\begin{lemma}
	\label{lem_conv}
	Let $X^{(n)}$, $n\in\mathbb N$, and $X$ be point processes on $\mathbb R$ such that $X^{(n)}\xrightarrow{vd}X$, and $X$ has a diffuse intensity measure. If $(I_n,n\in\mathbb N)$ is a decreasing sequence of intervals such that $I_n\downarrow I$ for some interval $I$, then
	\[\mathbb P\{X^{(n)}(I_n)\ge1\}\to\mathbb P\{X(I)\ge1\}.\]
\end{lemma}

\begin{proof}
	The proof is based on an application of the Skorokhod coupling (see, e.g., \cite{Res07}, p.~41). With the latter in mind, we may assume that $\bigl(X^{(n)},n\in\mathbb N\bigr)$ and $X$ are given on a common probability space, and $X^{(n)}\xrightarrow{v}X$ a.s. So,
	\begin{equation}
	\label{X_ineq}
	\bigl|X^{(n)}(I_n)-X(I)\bigr|\le
	\bigl|X^{(n)}(I)-X(I)\bigr|+\bigl|X^{(n)}(I_n)-X^{(n)}(I)\bigr|.
	\end{equation}
	Since the intensity measure of $X$ is assumed diffuse, $X(\partial I)=0$ a.s, and the first term on the right-hand side of \eqref{X_ineq} a.s.~vanishes as $n\to\infty$ due to the vague convergence. To deal with the second term, let us fix $N\in\mathbb N$. For all $n\ge N$,
	\[\bigl|X^{(n)}(I_n)-X^{(n)}(I)\bigr|=X^{(n)}(I_n\setminus I)\le X^{(n)}(I_N\setminus I),\]
	and the right-hand side converges a.s.~toward $X(I_N\setminus I)$. In other words,
	\[\limsup_{n\to\infty}\,\bigl|X^{(n)}(I_n)-X^{(n)}(I)\bigr|\le X(I_N\setminus I)\quad\text{a.s.}\]
	Letting $N\to\infty$ proves that the second term on the right-hand side of \eqref{X_ineq} a.s.~vanishes as $n\to\infty$ too.
	Hence, $X^{(n)}(I_n)\rightarrow X(I)$ a.s., and so $X^{(n)}(I_n)\xrightarrow{d} X(I)$, which clearly implies the claim.
\end{proof}

We may now proceed to the final part of the proof.
\begin{proof}[Proof of Theorem \ref{main_th}]
	In \eqref{nebound}, let us take $\varepsilon=n^{-\frac 15}$. Then,
	\begin{multline*}
	\mathbb P\bigl\{\xi_r^{(n)}[a,b]\ne\eta_r^{(n)}[a,b]\bigr\}\\\le
	c_rn^{-\frac 15}+
	\mathbb P\bigl\{\eta_r^{(n)}\bigl[a-n^{-\frac 15},a+n^{-\frac 15}\bigr]\ge1\bigr\}+
	\mathbb P\bigl\{\eta_r^{(n)}\bigl[b-n^{-\frac 15},b+n^{-\frac 15}\bigr]\ge1\bigr\}.
	\end{multline*}
	Lemmas \ref{lem_pois} and \ref{lem_conv} imply that, as $n\to\infty$,
	\begin{equation*}
	\mathbb P\bigl\{\eta_r^{(n)}\bigl[a-n^{-\frac 15},a+n^{-\frac 15}\bigr]\ge1\bigr\}
	\to\mathbb P\bigl\{\xi_r\{a\}\ge1\bigr\}=0,
	\end{equation*}
	and the same holds for $b$.	
	Together with the foregoing inequality, we have
	\begin{equation}
	\label{SR}
	\mathbb P\bigl\{\xi_r^{(n)}[a,b]\ne\eta_r^{(n)}[a,b]\bigr\}\to0,\quad n\to\infty.
	\end{equation}
	
	Let $\mathcal U$ stand for the ring of finite unions of bounded closed segments in $\mathbb R$. For each $U=[a_1,b_1]\cup\ldots\cup[a_l,b_l]\in\mathcal U$ and $k\in\mathbb N$, \eqref{SR} implies
	\begin{multline*}
	\bigl|\mathbb P\bigl\{\xi_r^{(n)}(U)=k\bigl\}-
	\mathbb P\bigl\{\eta_r^{(n)}(U)=k\bigl\}\bigr|\le
	\mathbb P\bigl\{\xi_r^{(n)}(U)\ne\eta_r^{(n)}(U)\bigr\}\\\le
	\sum_{i=1}^l\mathbb P\bigl\{\xi_r^{(n)}[a_i,b_i]\ne\eta_r^{(n)}[a_i,b_i]\bigr\}	
	\to0,\quad n\to\infty.
	\end{multline*}
	Moreover, by Lemma \ref{lem_pois},
	\[\lim_{n\to\infty}\mathbb P\bigl\{\eta_r^{(n)}(U)=k\bigl\}=\mathbb P\bigl\{\xi_r(U)=k\bigl\}.\]
	So, the last two formulas imply
	\[\lim_{n\to\infty}\mathbb P\bigl\{\xi_r^{(n)}(U)=k\bigl\}=\mathbb P\bigl\{\xi_r(U)=k\bigl\}.\]
	Notice that the process $\xi_r$ is simple, since it is a Poisson process with diffuse intensity measure (see, e.g., Proposition 6.9 in \cite{LP}).
	Hence, $\xi_r^{(n)}\xrightarrow{vd}\xi_r$ on $\mathbb R$ by Theorem 4.15 in \cite{Kal}.	
\end{proof}

In the sequel, we will need the following remark.

\begin{remark}
	\label{semi-c}
	Theorem \ref{main_th} remains true if we consider $\xi_r^{(n)}$ and $\xi_r$ as point processes on the semi-compactified real axis $\mathbb R\cup\{+\infty\}$, endowed with some relevant metric, say,
	$d(x,y)=|\mathrm e^{-y}-\mathrm e^{-x}|$. The proof follows along the same lines.
\end{remark}

\section{Implications of the main result}

From Theorem \ref{main_th}, we may easily deduce a number of known limit results which were often originally proved by direct complicated calculations. Moreover, this approach allows to obtain some far-reaching generalizations and infinite-dimensional extensions of those results. Finally, an application of Theorem \ref{main_th} often makes it possible to clarify some related surprising phenomena. As an example, consider $T_{r,m}^{(n)}$, $m\le n-1$, the first time when some $n-m$ (unspecified) of the $n$ coupon types have already arrived at least $r$ times each, and put $T_{r,m}^{(n)}=0$ for $m\ge n$. Various limit theorems for the case $r=1$ were studied in \cite{BB}, \cite{SC}, \cite{Hol81}, see also \S1.2 in \cite{KSC}. In particular, Theorem 4 in \cite{BB} asserts that
\begin{equation}
\label{chi2}
\ln 2n-\frac {T_{1,m}^{(n)}}n\xrightarrow{d}\ln Q,\quad n\to\infty,
\end{equation}
where $Q$ follows $\chi_{2m+2}^2$, the $\chi^2$-distribution with $2m+2$ degrees of freedom. Theorem \ref{appl1_th} below gives both the generalization for any $r\in\mathbb N$ and the infinite-dimensional extension in the sense of distributional convergence in $\mathbb R^{\infty}$, and also clarifies why a $\chi^2$-distribution appears.

Consider the random elements $V_r^{(n)}$ and $V_r$ in $\mathbb R^\infty$ given by
\begin{gather*}
V_r^{(n)}=\bigl(\psi_r^{(n)}\bigl(T_{r,m}^{(n)}\bigr),\,m\in\mathbb N\cup\{0\}\bigr),\\
V_r=\bigl(-\ln(r-1)!-\ln\sum_{j=1}^{m+1}E_j,\,m\in\mathbb N\cup\{0\}\bigr),
\end{gather*}
where $\psi_r^{(n)}$ is defined by \eqref{psi}, and $E_j$, $j\in\mathbb N$, are i.i.d.~$\mathtt{Exp}(1)$. 

\begin{theorem}
	\label{appl1_th}
	We have $V_r^{(n)}\xrightarrow{d}V_r$ in $\mathbb R^\infty$ as $n\to\infty$.
\end{theorem}

Before proving the theorem, we make a couple of important remarks.

\begin{remark}
	Restricting attention only to one-dimensional projections, we obtain
	\begin{equation}
	\label{T_rem}
	\frac {T_{r,m}^{(n)}}n-\ln n-(r-1)\ln\ln n\xrightarrow{d}-\ln(r-1)!-\ln S_{m+1},\quad n\to\infty,
	\end{equation}
	where $S_{m+1}=\sum_{j=1}^{m+1}E_j\sim\mathtt\Gamma(m+1,1)$. Confining ourselves only to the case $r=1$, we get
	\[\ln 2n-\frac{T_{1,m}^{(n)}}n\xrightarrow{d}\ln\bigl(2S_{m+1}\bigr),\quad n\to\infty.\]
	Finally, noting that $2S_{m+1}\sim\chi_{2m+2}^2$ leads to \eqref{chi2}.
\end{remark}

\begin{remark}
	Now allowing in \eqref{T_rem} for arbitrary $r\in\mathbb N$ but setting $m=0$, we can easily deduce \eqref{ER}, the limit theorem by Erd\H{o}s and R\'enyi.
\end{remark}

We now turn to the proof.

\begin{proof}[Proof of Theorem \ref{appl1_th}]
	We will need some additional notation. Fix $m\in\mathbb N\cup\{0\}$, and let
	$$L^m\colon M_p\bigl(\mathbb R\cup\{+\infty\}\bigr)\to\bigl(\mathbb R\cup\{-\infty,+\infty\}\bigr)^{m+1}$$ map $\sum\delta_{x_i}$ into the vector of its \lq\lq last-but-$j$\rq\rq th points (possibly infinite), $0\le j\le m$. In other words,
	$L^m(\mu)=\bigl(L_j(\mu),0\le j\le m\bigr)$, where
	\[L_j(\mu)=\inf\{x\in\mathbb R\colon\mu(x,+\infty]=j\}.\]
	Further, we denote by $\mu|_K$ the point measure $\mu$ restricted to the compact set $K\subset\mathbb R\cup\{+\infty\}$.
	
	By Theorem \ref{main_th} in the form of Remark \ref{semi-c}, using the Skorokhod coupling, we may consider $\bigl(\xi_r^{(n)},n\in\mathbb N\bigr)$ and $\xi_r$ on a common probability space and assume that $\xi_r^{(n)}\xrightarrow{v}\xi_r$ on $\mathbb R\cup\{+\infty\}$ a.s. Hence, by Proposition 3.13 in \cite{Res87}, $L^m\bigl(\xi_r^{(n)}|_{[a,+\infty]}\bigr)\to L^m(\xi_r|_{[a,+\infty]})$ a.s.~for any $a\in\mathbb R$. This implies $L^m\bigl(\xi_r^{(n)}\bigr)\to L^m(\xi_r)$ a.s., and so 
	\begin{equation}
	\label{fd_conv}
	L^m\bigl(\xi_r^{(n)}\bigr)\xrightarrow{d}L^m(\xi_r).
	\end{equation}
	
	Note that, by \eqref{xin_def},
	\[L^m\bigl(\xi_r^{(n)}\bigr)=\bigl(L_j\bigl(\xi_r^{(n)}\bigr),0\le j\le m\bigr)=\bigl(\psi_r^{(n)}\bigl(T_{r,j}^{(n)}\bigr),0\le j\le m\bigr),\]
	and is thus just the projection of $V_r^{(n)}$ onto the first $m+1$ coordinates (from $0$-th to $m$-th).
	On the other hand, by Remark \ref{nhg_PPP} and
	the i.i.d.~property of inter-arrival times for $\xi$, we come to
	\[L^m(\xi_r)=(L_j(\xi_r),0\le j\le m)=\Bigl(-\ln(r-1)!-\ln\sum_{k=1}^{j+1}E_j,
	0\le j\le m\Bigr),\]
	which is, similarly, the projection of $V_r$ onto the first $m+1$ coordinates.
	Hence, \eqref{fd_conv} proves the finite-dimensional convergence $V_r^{(n)}\xrightarrow{fd}V_r$ in $\mathbb R^{\infty}$ as $n\to\infty$. To complete the proof, it only remains to recall that in $\mathbb R^{\infty}$ the notions of finite-dimensional convergence and convergence in distribution are equivalent (see, e.g., \cite{Res07}, pp.~53--54).
\end{proof}

We now consider another application of Theorem \ref{main_th}, namely, to \lq\lq rare\rq\rq\ coupon types. Recall that an arriving coupon is assumed to belong to any of the $n$ types with the same probability $\frac 1n$. But, due to random factors, some coupon types will need a long time until they arrive for the $r$-th time. Taking \eqref{ER} into account, we will call a type $i\in\mathbb N_n$ $x$-rare, $x\in\mathbb R$, if
\begin{equation}
\label{rare}
Y_{i,r}^{(n)}\ge nx+n\ln n+(r-1)n\ln\ln n.
\end{equation}
Denote by $C_r^{(n)}(x)$ the number of $x$-rare types:
\[C_r^{(n)}(x)=\sum_{i=1}^n\mathds1\bigl\{Y_{i,r}^{(n)}\ge nx+n\ln n+(r-1)n\ln\ln n\bigr\}.\]
Below we state and prove a functional limit theorem for $C_r^{(n)}=\bigl(C_r^{(n)}(x),x\in\mathbb R\bigr)$ in the Skorokhod $J_1$-topology.

Let $N=\bigl(N(t),t\ge0\bigr)$ be a homogeneous unit-rate Poisson process, considered not as a random point measure, but classically, as a L\'evy process with Poisson increments. Actually, $N(t)=\xi(0,t]$, where $\xi$ is introduced in Remark \ref{nhg_PPP}. Also, let $N_r(x)=N\bigl(\frac{\mathrm e^{-x}}{(r-1)!}\bigr)$, $x\in\mathbb R$.

\begin{theorem}
	\label{rare_th}
	We have $C_r^{(n)}\xrightarrow{d}N_r$ in $D(\mathbb R)$, endowed with the $J_1$-topology, as $n\to\infty$.
\end{theorem}

\begin{remark}
	Strictly speaking, the processes $C_r^{(n)}$ and $N_r$ are c\`agl\`ad, not c\`adl\`ag, and we should have considered $D_{\rm left}(\mathbb R)$, the space of left-continuous functions with finite right limits, instead of $D(\mathbb R)$. But as the $J_1$-topology may be introduced on $D_{\rm left}(\mathbb R)$ in the same way as on $D(\mathbb R)$, we will close our eyes to these differences (cf.~Remark 3.2 on p.~58 in \cite{Res07}).
\end{remark}

\begin{remark}
	For fixed $x$, the distributional convergence of $C_r^{(n)}(x)$ to $N_r(x)$, as well as its rate in terms of total variation distance, were known before as a result of Stein-Chen method for Poisson approximation (see, e.g., Chapter 6 in \cite{BHJ} and Example 4.34 in \cite{Ross}). For related finite-dimensional results see also Theorem 1 on p.~172 in \cite{KSC} and Theorem 2 in \cite{Bol}.
\end{remark}

\begin{proof} [Proof of Theorem \ref{rare_th}]
	Let $h^\leftarrow(x)=\frac{\mathrm e^{-x}}{(r-1)!}$, $x\in\mathbb R$, be the inverse function to $h$ given by \eqref{h}.
	Theorem \ref{main_th} and Remark \ref{nhg_PPP} imply that, by the continuous mapping theorem, $\xi_r^{(n)}\circ(h^\leftarrow)^{-1}\xrightarrow{vd}\xi$. (The intuitively obvious continuity of $T\colon M_p(\mathbb R)\to M_p\bigl((0,+\infty)\bigr)$ given by $T(\mu)=\mu\circ(h^\leftarrow)^{-1}$ follows from Proposition 3.18 in \cite{Res87}.) According to Theorem 4.20 in \cite{Kal}, the distributional convergence of random point measures on $(0,+\infty)$ in the vague topology is equivalent to that of the associated cumulative processes in the $J_1$-topology. So, $\xi_r^{(n)}\bigl((h^\leftarrow)^{-1}(0,\cdot]\bigr)\to N(\cdot)$ in the latter sense, and thus, by transfer, $\xi_r^{(n)}[\cdot,+\infty)\to N_r(\cdot)$. It remains to note that, by \eqref{xin_def}, \eqref{psi}, and \eqref{rare}, $\xi_r^{(n)}[x,+\infty)=C_r^{(n)}(x)$ for any $x\in\mathbb R$.
\end{proof}

\end{document}